\documentclass[11pt,leqno]{article}

 \usepackage[all]{xy}
\usepackage{amssymb,amsmath,amsthm,  url,rotating}
\usepackage{setspace}
 \usepackage{graphicx,epsfig}
 \usepackage{xcolor,graphicx}
  \usepackage{makeidx}
  \usepackage[all]{xy}
\usepackage{amssymb,amsmath,amsthm,  url,rotating}
\usepackage{setspace}
\usepackage{mathrsfs}
 \usepackage{graphicx,epsfig}
 \usepackage{xcolor,graphicx}
  \usepackage{makeidx}
\newcommand{\n}{\noindent}

\newcommand{\vp}{\varepsilon}
\newcommand{\bb}[1]{\mathbb{#1}}
\newcommand{\cl}[1]{\mathcal{#1}}

\newcommand{\ovl}{\overline}

\theoremstyle{plain}
\newtheorem{thm}{Theorem}[section]
\newtheorem{lem}[thm]{Lemma}

\newtheorem{pro}[thm]{Proposition}

\newtheorem{cor}[thm]{Corollary}

\theoremstyle{definition}

\theoremstyle{remark}
\newtheorem{rem}[thm]{Remark}

\numberwithin{equation}{section}

\def\tilde{\widetilde}

\renewcommand{\tilde}{\widetilde}

\def\C{\cl C}
\def\CC{\bb C}

\def\F{\bb F}

\def\Z{\bb Z}
\def\d{\delta}

\def\CC{\bb C}

\def\CC{\bb C}

\def\F{\bb F}

\def\d{\delta}

\def\CC{\bb C}

\def\phi{\varphi}

\def\n{\noindent}

\def\a{\alpha}

\makeindex

\def\C{\mathscr{C}}

\def\I{\cl  I}

  \def\a{\alpha}

\begin{document}

 \pagenumbering{roman}

 \thispagestyle{empty}

%\vfill\eject

%\doublespacing
%\tableofcontents
%\singlespacing
%\tableofcontents

\setcounter{page}{1}

\newtheorem{ass}{}

 \pagenumbering{arabic}
 \def\ovl{\overline}
 \def\U{\cl U}

      \def\l{\ell}
      
       \title{A note on strong similarity and the Connes embedding problem}

\author{by\\
Gilles  Pisier\footnote{ORCID    0000-0002-3091-2049}   \\
Sorbonne Universit\'e\\
and\\ Texas  A\&M  University}

 \maketitle
      
      \begin{abstract} We show that there exists a completely bounded (c.b. in short) unital homomorphism $u$
from a $C^*$-algebra $C$ with the lifting property (in short LP) into a QWEP von Neumann algebra $N$
that is not strongly similar to a $*$-homomorphism, i.e.
      the similarities that ``orthogonalize" $u$  (which exist since $u$ is c.b.)  cannot belong  to
       the von Neumann algebra $N$.
       Moreover, the map $u$ does not admit any c.b. lifting up into the WEP $C^*$-algebra
       of which $N$ is a quotient.  We can take   $C=C^*(\F_\infty)$ the full $C^*$-algebra of the free group $\F_\infty$ with 
infinitely many generators   and $N= B(H)\bar \otimes M$ where $M$ is the 
von Neumann algebra generated by the reduced $C^*$-algebra of $\F_\infty$. Incidentally we observe an analogue
for strong similarity  of Haagerup's (and Paulsen's) similarity formula for the cb-norm  : if $C$ is any
unital $C^*$-algebra and $N$ any von Neumann algebra then for any bounded unital homomorphism
$u: C \to N$ we have
$$\|u\|_{mb}= \inf\{ \|S\|\|S^{-1}\| \}$$
where the inf (which is attained) runs over all invertible $S\in N$ such that  $S u(.) S^{-1}$ is   a $*$-homomorphism. A similar formula holds for the mb-norm of  derivations. We end the note by a quick proof of the main point
using the mb-norm and the space $R_n\cap C_n$.
  \end{abstract}

       In this note (which is not intended for publication) we   produce 
      a unital c.b. homomorphism $u$ from a $C^*$-algebra $C$ with LP
      (which we could choose to be $C^*(\F_n)$ if we wish using free Haar unitaries in place of $(x_j)$)
      into a QWEP von Neumann algebra $N$
      which is NOT  strongly similar to a $*$-homomorphism,
      here by ``not strongly similar" we mean that the similarity cannot be chosen
      in the von Neumann algebra $N$.
      
      This means that a certain statement from \cite{[19]} crucially used in \cite{fang} 
      is not valid, as well as the application proposed in    \cite{fang}  to     the Connes embedding problem.
      
       If $N=B/I$ with $B$ WEP ($I$ closed two sided ideal in $B$)
      the map $u$ does not admit any c.b. lifting. 
      Indeed the paper  \cite{fang} shows that if it did then
      $u$ would be strongly similar to a $*$-homomorphism.
      
      \begin{rem} By known results (see \cite{oz}) there are 
      c.b. maps $u: C \to B/I$  that do not admit any c.b. lifting.
      Apparently the existence of a c.b. homomorphism is new.
      \end{rem}
      
      \section{Examples of c.b. homomorphisms on $C$  (e.g. on  $C^*(\F_n)$)}\label{s1}
      Our goal is to produce a c.b.  unital homomorphism
       $u: C \to {\cl N}$ 
       from a 
       $C^*$-algebra $C$ with LP into a QWEP von Neumann algebra ${\cl N}$
       that is not   strongly similar to a $*$-homomorphism and does not admit
       any c.b. lifting. This contradicts \cite[Lemma 2.3]{[19]}.
%\begin{proof}

We will use $\cl N=M_2(N)$ where $N=B(H)\bar \otimes M$
(von Neumann algebraic tensor product)   with $M$ being the 
von Neumann algebra of the free group $\F_n$ with $1\le n\le \infty$ generators.
As for $C$ we use the full free product of a family  of $n$  copies of the $C^*$-algebra
generated either by a single semicircular  (i.e. $C([-2,2])$) or by a single 
unitary generating $C^*(\Z)$. In the latter case $C=C^*(\F_n)$  (see Remark \ref{Z}).

      Let $\ell_j$ (resp.  $r_j$)
      denote the left (resp. right) creators on the full Fock space $H$
      so that 
      $$x_j=\l_j+\l_j^* \quad (1\le j\le n)$$
      are free semicirculars in Voiculescu's sense (see \cite{VDN}) generating the  von Neumann algebra  $M\subset B(H)$
       and their commutant $M'$
       is generated
       by    $$y_j=r_j+r_j^*. $$
       See \cite{VDN} for full details and missing background on this topic.
       Let $C=\ast_j C^*(x_j)$ (full free product).
       Let $\sigma: C \to M$ be the natural representation, so that $M=\sigma(C)''$,
       and let $\pi: C \to M \bar \otimes B(H)$ be defined by $\pi(x)=\sigma(x)\otimes 1$.
       
       Let $N=B(H)\bar \otimes M$.
              We will produce a c.b. derivation $\d : C \to  N$
       relative to $\pi$
       with $\|\d\|_{cb} \le 2$ but such that any $T$ in $N$
       such that $\d(x)=[T, \pi(x)]$ satisfies 
       $\|T\|\ge \sqrt n/4$.

       Then we will show that the associated homomorphism
       $u: C \to M_2(N)$
       defined by
       $$C\ni x \mapsto \left( \begin{matrix} \pi(x)  & \ \d(x)\\ 0& \ \pi(x )   \end{matrix}  \right)$$ 
       is also c.b. but if $n=\infty$ it is not strongly similar to a $*$-homomorphism.
       An estimate will also be    given for large finite $n$'s.

       The construction we use is a particular case of the one described in \cite{toho}.
       
       Let $$\Delta= \sum  - r_j \otimes \l_j^*+r_j^* \otimes \l_j \in B(H\otimes H)$$
       
       Clearly $\| \sum \l_j \l_j^* \|= \| \sum r_j r_j^* \| \le 1$ and hence
       $$\|\Delta\|\le 2.$$
       Let $C=\ast_j C^*(x_j)$ (full free product).
       We denote by $\pi:  C\to M \otimes 1$ the representation from the full free product to $M$
       taking $x_j$   to  $x_j \otimes 1$. Equivalently $\pi(x)=\sigma(x) \otimes 1$ for any $x\in C$.
       
       Note that $C$  or $C^*(\F_n)$ has the LP. This is   due to Kirchberg, see  \cite{157} for more precise references. More generally,
       the stability of the LP under full free products is due to Boca \cite{boca}.
       
       Let $\d : C \to B(H\otimes H)$ be defined by
       $$\d(x)= [\sigma(x) \otimes 1 , \Delta]=[\pi(x)   , \Delta] .$$
       Then $\d $ is inner (with respect to  $\pi$)  and   \begin{equation}  \label{e0}
       \|\d\|_{cb} \le 2.\end{equation}
   Note $[\sigma(x_i),y_j]=0$ and hence  $M=\sigma(C)''$ commutes with $y_j=r_j+r_j^*$.
  Therefore $[\sigma(x) , r_j+r_j^*]=0$ and hence     
       $$ \d(x)=  [\pi(x) , \sum r_j^* \otimes \l_j^* +r_j^* \otimes \l_j] = 
       [\pi(x) , \sum r_j^* \otimes x_j  ] = \sum [\sigma(x)  ,   r_j^*] \otimes x_j   = -\sum [\sigma(x)  ,   r_j ] \otimes x_j ,$$
       which we record as
       \begin{equation}  \label{e12}
      \forall x\in C\quad   \d(x)= -\sum [\sigma(x)  ,   r_j ] \otimes x_j 
        \end{equation}
    This shows that   $\d(x)\in B(H) \otimes E$ where $E$ is the linear span
    of $(x_j)$ when $n<\infty$. In any case
   the
     von Neumann algebra  generated by $\pi(C)$ and $\d(C)$ is included in $B(H)\bar\otimes \{x_j\}'' =B(H)\bar {\otimes} M$.
       
       Let $T \in B(H)\bar {\otimes} M$  be such that for all $x\in C$ $$\d(x)= [\pi(x) , T].$$
       Note that such a $T$ exists, for example  $T=-\sum r_j \otimes x_j$ does by \eqref{e12}, but we will see
       that the norm of any such $T$ must be large.\\
 There is a projection $P$ on $M$ with range the span of $(x_j)$ with $\|P\|_{cb} \le 2$. Indeed
 this is due to Haagerup in  the free group case, see e.g. \cite[p. 184]{100},
 and the semicircular case is similar see e.g. \cite[p. 209]{100}.\\
     Therefore  we must have
       $$\d(x)= (Id \otimes P) \d(x)=[\pi(x) , T_1]$$   
       with
       $$T_1= (Id \otimes P)(T).$$
       It follows that $T_1=\sum z_j \otimes x_j$ for some $z_j\in B(H)$ and since we have
        $$\max\{\| \sum b_j b_j^* \|^{1/2} , \| \sum b_j^* b_j \|^{1/2}\}    \le \|\sum b_j \otimes x_j\|$$ 
        for all $b_j\in B(H)$ (see e.g.    \cite[(9.9.9) p. 208]{100}) we obtain
       \begin{equation}  \label{e1}
        \max\{\| \sum z_j z_j^* \|^{1/2} , \| \sum z_j^* z_j \|^{1/2}\}    \le \|T_1\|\le \|P\|_{cb}\|T\|  \le  2\|T\| .
       \end{equation} 
    Moreover since 
       $$\d(x)= \sum [\sigma(x)  ,   r_j^*] \otimes x_j = - \sum [\sigma(x)  ,   r_j] \otimes x_j =
       \sum [\sigma(x)  ,   z_j] \otimes x_j $$
       we must have
       $$ r_j^*-z_j \in M' \text{  and  }  r_j+z_j \in M' $$
       
       $$y_j=  (r_j +z_j )+ (r_j^*-z_j)  $$
       and hence
       $$\sqrt n=(\sum \|y_j \Omega\|^2)^{1/2} \le   (\sum \|(r_j +z_j ) \Omega\|^2)^{1/2} + (\sum \|(r_j^*-z_j)  \Omega\|^2)^{1/2}  $$
       but  by the trace identity in $M'$
       $$(\sum \|(r_j +z_j ) \Omega\|^2)^{1/2}
       =(\sum \|(r_j +z_j )^* \Omega\|^2)^{1/2} = (\sum \| z_j ^* \Omega\|^2)^{1/2}$$
       and clearly $(\sum \|(r_j^*-z_j)  \Omega\|^2)^{1/2}  = (\sum \| z_j  \Omega\|^2)^{1/2}$  
       so that we obtain by \eqref{e1}
        \begin{equation}  \label{e2}\sqrt n \le  \| \sum z_j z_j^* \|^{1/2} +\| \sum z_j^* z_j \|^{1/2} \le 2 \|T_1\| \le 4\|T\|.\end{equation}
        In other words, denoting $\d_T(x)=[\pi(x),T]$ we have
        \begin{equation}  \label{e13} \sqrt n /4 \le \inf \{ \|T\| \mid T\in N, \d =\d_T \} \le \sqrt n+1,\end{equation} 
        where the last inequality follows from \eqref{e12} and $\|\sum r_j \otimes x_j \|\le  \sqrt n+1$.
        
         Let $\d$ be as before. Then
    $$C\ni x \mapsto u(x)=\left( \begin{matrix} \sigma(x)\otimes 1 & \ \d(x)\\ 0& \ \sigma(x)\otimes 1  \end{matrix}  \right)$$ 
    is a c.b. homomorphism $u$ into $\cl N=M_2( N)$. \\
    We have $\|u\|_{cb} \le 1 + \|\d\|_{cb} \le 3$ by \eqref{e0} (and replacing $\d$ by $(\vp/2)\d$ 
    we could get $\|u\|_{cb}\le 1+\vp$ if desired).

    \begin{lem} Let $S$ be an invertible in $\cl N=M_2( N)$ such that
    $x \mapsto S u(x) S^{-1}$ is a $*$-homomorphism. Then there is 
      $T$ in $N$ with
    $$2\|T\|  \le \|S\|\|S^{-1}\| $$ 
    such that $\d(x)=[x,T]$ for any $x\in C$.

    \end{lem}
    \begin{proof} Here we use Theorems  \ref{t0} and \ref{t2}.
    Obviously $\|\d\|_{mb} \le \|u\|_{mb}$ (and also $\|u\|_{mb}\le 1+\|\d\|_{mb}$). This means
    $$2\min\{ \|T\| \mid T\in N,\   \d=\d_T\}\le \inf \{  \|S\|\|S^{-1}\| \mid  S\in M_2(N)\ \  u=u_S\} .$$ 
         The announced result follows.\\
    See also \cite[p. 80]{95} for a more ``concrete" proof
    of the weaker estimate $\|T\|  \le (\|S\|\|S^{-1}\|)^2$
    where this is proved following
    an  argument of Paulsen producing $T$ more explicitly.
      \end{proof}

By \eqref{e2}  if $n<\infty$ we have
    \begin{equation}  \label{e14} n^{1/2} /2 \le  \inf \{\|S\|\|S^{-1}\| \}   =\|u\|_{mb}\end{equation}
    where the inf runs over all invertible $S\in \cl N$ such that
    $x \mapsto S u(x) S^{-1}$ is a $*$-homomorphism,
    and if we use $n=\infty$ to define $\d$ and $u$, then  \eqref{e2} shows that  there are no such invertible $S\in \cl N$. 
        
        Now it remains to check that
        $\cl N=M_2(B(H) \bar \otimes M)$    is QWEP, i.e.
        that $\cl N=B/I$ with $B$ unital WEP $C^*$-algebra and $I\subset B$ ideal (two-sided closed 
        self-adjoint). 
         But this reduces clearly to the fact that $M$ is QWEP
         which is  exactly the original observation of Connes
        in the 1976 paper where he raised what is now known 
        as the Connes embedding problem.

This completes the first part of our goal.  In Corollary \ref{c1}
we   show the second part : that $u: C \to \cl N=B/I$ does not admit any c.b. lifting.

       \begin{rem} The $*$-homomorphism $\pi$ maps $C$ onto
       the $C^*$-algebra $C^*_\pi =C^*_\sigma \otimes 1 \subset B(H)\bar\otimes M $. 
       Without loss of generality we could describe our example as a derivation
       from $C^*_\pi=C^*_\sigma \otimes 1$ to $B(H)\bar\otimes M $. 
       Let $\I=\ker(\pi)$.
       Since for any other $C^*$-algebra $D$ we have $$C^*_\pi \otimes_{\max} D=
       [C \otimes_{\max} D]/ [\I \otimes_{\max} D]$$ 
       the mb-norm of the derivation (or its ``strong innerness") remains the same in either viewpoint.
   \end{rem}   
   
       \begin{rem} Note in passing (although this is not needed) that the     von Neumann algebra  $N_1$ generated   by
        $\pi(C) $   (or $M\otimes 1 $)  and $\d(C) $  coincides with  $N=B(H) \bar \otimes M$.  
        Indeed, recall
        $$\d(x)= \sum [\sigma(x)  ,   r_j^*] \otimes x_j = - \sum [\sigma(x)  ,   r_j] \otimes x_j $$
                Observe that
        $[l_k, r_j]=[l_k^*, r_j^*]=0 $ for all $k,j$ and also  $[l_k, r_j^*]= 0$ if $k\not=j$
        and $[l_k, r_k^*]= P_\Omega$. Thus
        $$\d(x_k)= P_\Omega \otimes x_k$$
        which implies that $N_1$ contains
        all elements of the form $a P_\Omega b \otimes x_k$ ($a,b\in M$), i.e.
        all elements in $F(H) \otimes x_k$ where $F(H)$ denotes the space of finite rank operators on $H$.
        It follows that $N_1$ contains $F(H) \otimes M $
        and hence $N_1\supset B(H) \bar \otimes M$. Since we already noted that the converse
        is obvious, this proves as announced $N_1=B(H) \bar \otimes M$.
        \end{rem}
       
    We know that $\cl N=M_2( N) $ is QWEP,
    so we have $\cl N =B/I$ with $B$ WEP.
       The following, 
    which actually is an interesting fact, 
    shows that   \cite[Lemma 2.3]{[19]} is not correct :
    
    \begin{cor}\label{c1} The above c.b. homomorphism $u: C \to \cl N=M_2( N)$, which   
    has
    $$\|u\|_{cb} \le 3$$   is such that
    any c.b. lifting $\hat u: C \to B$ satisfies
    $$ \sqrt n/2\le \|\hat u\|_{cb} .$$
    Thus if we work with $n=\infty$ we obtain a $u$ without any c.b. lifting.
   \end{cor}
    \begin{proof} 
     We will use freely
    Kirchberg's characterization of the WEP by tensor products
    as described in \cite[\S 9.3]{154}. Let $q: B \to B/I$ denote the quotient map.
 Let  $\hat u$ be a c.b. lifting of $u$, i.e. $q\hat u=u$.
We have (here $\C=C^*(\F_\infty)$)
    $$\|Id_{\C} \otimes \hat u : \C \otimes_{\min} C  \to \C \otimes_{\min} B \|\le \|\hat u\|_{cb} $$
    and since $B$ WEP implies $\C \otimes_{\min} B=\C \otimes_{\max} B$,
    we   have
    $$\|Id_{\C} \otimes \hat u : \C \otimes_{\min} C  \to \C \otimes_{\max} B \|\le \|\hat u\|_{cb} $$
    and {\it a fortiori}
    $$\|Id_{\C} \otimes \hat u : \C \otimes_{\max} C  \to \C \otimes_{\max} B \|\le \|\hat u\|_{cb} .$$
    Since $u=q\hat u$, we also have (recall $\cl N  =B/I$) 
    $$\|Id_{\C} \otimes   u : \C \otimes_{\max} C  \to \C \otimes_{\max} {\cl N} \|\le \|\hat u\|_{cb} .$$
    With the notation in \cite{157} (see \S \ref{s2} below) this means
    $$\|  u\|_{mb} \le \|\hat u\|_{cb} ,$$
    which implies since $\|  \d\|_{mb} \le \|  u\|_{mb}$
        $$\|  \d\|_{mb} \le \|\hat u\|_{cb} .$$
    By \eqref{e14} we obtain $\sqrt n /2  \le \|\hat u\|_{cb} .$
 \end{proof}
    
    \begin{rem}\label{Z} [Free group variant] Let $(g_j)$ be the free generators of a free group $G$.
    Let $\lambda$ and $\rho$ denote the left and right regular representations on $G$.
    It is well known (see \cite{toho})  that $\rho(g_j)$  can be written as $a_j+b_j$ 
    with $\sum a_j^* a_j\le 1$ and $\sum b_j b_j^*\le 1$ (and of course similarly for $\lambda(g_j)$).
    The preceding construction    runs just the same with $C=C^*(G)$ if   one takes
     for $ M$ the von Neumann algebra 
    generated by $\lambda$ and one sets
    $\d(x)=\sum [x,a_j] \otimes \lambda(g_j) $ for $x\in M$. In other words we may replace 
    $(x_j)$ (free gaussians) by $\lambda(g_j)$ (free Haar unitaries) with no essential change.
    \end{rem}
    
    \section{The mb-norm and strong similarity}\label{s2}
    
      The following result   was observed in \cite[(2) p. 81]{toho}. Similar ideas appear
           in  \cite{fang}.
        
        \begin{pro} \label{p1}
        Consider $u: C \to N  $ a unital  c.b.  homomorphism from a $C^*$-algebra $C$
        with values in a von Neumann algebra  $N\subset B(H)$.
        Let $v: N' \to B(H)$ be the inclusion map.
        Then $u$ is strongly similar to a $*$-homomorphism
        iff the map $v . u : N' \otimes  C\to B(H) $ (defined by
        $v . u (a\otimes b)= v(a)u(b) $ extends to
        a c.b. map $w$ from $N' \otimes_{\max} C$ to $B(H)$.
        Moreover
         \begin{equation}  \label{e9}
         \| w:  N'\otimes_{\max} C \to B(H)\|_{cb} =\inf\{\|S\|\| S^{-1}  \| \}\end{equation}
        where the (attained) inf runs over all invertible  $S\in N$
        such that  $S u(.) S^{-1}$ is   a $*$-homomorphism.
\end{pro}
          \begin{proof}  Assume $w$ c.b. Clearly $w$ is still a homomorphism,
          and hence if it is c.b. it must be similar to a $*$-homomorphism
          whence the existence of an invertible $S$ on $H$ (with
          $\|S\|\| S^{-1}  \| =\|w\|_{cb} $)
          such that $x\mapsto S w(x) S^{-1} $ is a $*$-homomorphism.
          But restricting this to $N' \otimes 1$ we find that 
          $N' \ni y \mapsto  S y S^{-1}$ is  a $*$-homomorphism
          which implies by a simple calculation 
          ($S y S^{-1}   = (S y^* S^{-1})^* $ and hence $S^*S y=yS^*S$)
          that $|S| \in N''=N$.
          Now $S=U|S|$ with $U$ unitary and the map
          $C\ni x \mapsto  U (|S| u(x) |S|^{-1}) U^{-1}$
          is  a $*$-homomorphism.
          Since $U$ is unitary $C\ni x \mapsto    (|S| u(x) |S|^{-1})  $
           is  a $*$-homomorphism, proving the strong similarity of $u$.
           
           Conversely, if $S\in N$ is such that $C\ni x \mapsto    (S u(x) S^{-1})  $
           is  a $*$-homomorphism denoted by $\rho: C \to N$,
           then obviously $v.\rho $ is a (completely contractive)  
           $*$-homomorphism from $ N'\otimes_{\max} C$ to $B(H)$.
           Then  $a\otimes b \mapsto S^{-1} \rho.v (a\otimes b)   S $
           has cb-norm at most $\|S\|  \| S^{-1}  \|$ on $ N'\otimes_{\max} C$ and since $S$ commutes with $v(b)$, we have
           $S^{-1} v.\rho (a\otimes b)   S = v.u(a\otimes b) $. Thus we 
           conclude that $$\| v.u: N'  \otimes_{\max} C\to B(H)\|_{cb} \le \inf\|S\|  \| S^{-1}  \|.$$
           The converse inequality was proved in the first part. There we also proved that the inf is actually a minimum.
    \end{proof}

   In \cite{157} we  introduced (and extensively studied) the mb-norm of a map
   $u: E \to B$
    from a subspace of a $C^*$-algebra  $E \subset A$ (i.e. a so-called operator space)
   to a $C^*$-algebra $B$.\\
   The mb-norm is defined as follows
   $$\|u\|_{mb} = \| Id_\C \otimes u: (\C \otimes_{alg} E, \|\ \|_{\C \otimes_{\max}  A}   ) \to \C \otimes_{\max} B  \| $$
   where $(\C \otimes_{alg} E, \|\ \|_{\C \otimes_{\max}  A}   )$ denotes the algebraic tensor product
   of $\C$ and $E$ equipped with the norm induced on it by ${\C \otimes_{\max}  A}$.
   We denote (somewhat abusively) by $\C \otimes_{max} E$ the completion of the latter normed space. When $E=A$ we recover the usual 
   $C^*$-algebraic max-tensor product. 
   
   One striking feature
 of the mb-norm that diverges from the cb-case is that if $B$ is an {\it arbitrary}  $C^*$-algebra
 any linear map $u: E \to B$ with   $\|u\|_{mb}=1$ ``extends" to a map
 $\tilde u: A \to B^{**}$  with $\|\tilde u\|_{mb}=1$.
 Therefore if $B$ is an {\it arbitrary}
   von Neumann algebra any  $u: E \to B$ with   $\|u\|_{mb}=1$  extends to a map
 $\tilde u: A \to B$  with $\|\tilde u\|_{mb}=1$. See \cite[\S 6.3]{154} for more on this theme, which originates in Kirchberg's work on decomposable mappings.

   As observed in \cite{157} the properties of the max-tensor product imply that for any $C^*$-algebra
    $D$ we have
     \begin{equation}  \label{e10} \| Id_D \otimes u: D \otimes_{max} E  \to D \otimes_{\max} B  \|_{cb} \le   \|u\|_{mb} . \end{equation}
    
    Following partial results due to Hadwin and Wittstock,
     Haagerup proved  in \cite{Ha} that the cb-norm of a unital homomorphism
    $u: C \to B(H)$ on a $C^*$-algebra $C$ is equal to the (attained) infimum of $  \|S\|\|S^{-1}\|$
 over all invertibles $S \in B(H)$ such that $S u(.) S^{-1}$ is   a $*$-homomorphism.
 (see   \cite[Cor. 4.4, p. 79]{95} for more on this).

    The following identity is the analogous fact for
    the mb-norm of $u$:
    \begin{thm}\label{t0} Let $u: C \to {  N}$ be a unital homomorphism from a unital $C^*$-algebra $C$ to a von Neumann algebra ${  N}\subset B(H)$.
    Then:
    $$\|u\|_{mb} =\inf\{\|S\|\| S^{-1}  \| \}$$
    where the (attained) inf runs over all invertibles $S \in {  N}$
    such that $S u(.) S^{-1}$ is   a $*$-homomorphism.
    \end{thm}
    
    \begin{proof} By \eqref{e10}
    with $D={N}'$ and $E=C$  
    and by \eqref{e9}  there is $S\in N$ with
   $ \|S\|\| S^{-1}  \|\le \|u\|_{mb}$
    such that $u_S(.)=S u(.) S^{-1}$ is   a $*$-homomorphism into $N$.
    Conversely, since a  $*$-homomorphism has unit mb-norm,
    and left or right multiplication by an element  $x\in N$
    has mb-norm at most $\|x\|$, the fact that
     $u(.)=S^{-1}u_S(.)S$ with   $S\in N$ gives us
    $\|u\|_{mb} \le  \|S\|\| S^{-1}\| $.
     \end{proof}
     
     Haagerup's result was extended by Paulsen \cite{Pa}
     to the {\it non self-adjoint} case. The analogue for the mb-norm
     (with the same argument as for the preceding statement) is as follows. At this point the converse part is not clear.
    
      \begin{thm}\label{t1} Let $u: C \to {  N}$ be a unital homomorphism from a  unital operator algebra $C\subset B(\cl H)$  to a von Neumann algebra ${  N}\subset B(H)$.
  If $\|u\|_{mb}<\infty$ there is an invertible $S \in {  N}$
  with $\|S\|\| S^{-1}  \|\le \|u\|_{mb}   $
       such that $S u(.) S^{-1}$ is   a  unital completely contractive homomorphism.
    \end{thm}
      
     As for derivations,
       the following analogous formula was suggested by Jean Roydor:
     
     \begin{thm}\label{t2} Let $M\subset B(H)$ be a von Neumann algebra
     and $A\subset M$ a $C^*$-algebra. Let 
  $\d: A \to M$ be a derivation.
  Then $\|\d\|_{mb} <\infty $ iff there is $T\in  M$ such that
  $\d=\d_T$ and we have
  $$\|\d\|_{mb} =2\inf\{ \|T\| \mid T\in  M, \ \d=\d_T\},$$
  where the inf is attained.\\
  Moreover
  $\|\d\|_{mb}=\|D\|_{cb}$ where $D: M' \otimes_{\max} A \to B(H)$ is the 
  mapping   defined by $D(x\otimes a)=x\d(a)$.
 
     \end{thm}
      \begin{proof}  
      That the inf is attained is  immediate
      by the weak* compactness of the closed balls
      in $M$.  Assume $\|\d\|_{mb} <\infty $.
      Then by \eqref{e10} we have
      $\|Id \otimes \d :  M' \otimes_{\max} A \to M' \otimes_{\max} M\|_{cb} \le \|\d\|_{mb}$
      and hence $\|D\|_{cb} \le \|\d\|_{mb}$.
      Let $\pi: M' \otimes_{\max} A \to B(H)$ be the continuous $*$-homomorphism
      defined by $\pi(x\otimes a)=xa $. It is easy to check that
      $D$ is a $\pi$-derivation. By Christensen's theorem (see \cite[p. 96]{95}) 
      there is $T\in B(H)$ with $\|T\|\le (1/2)\|D\|_{cb}$ such that $D(z)=[\pi(z), T]$ for any $z\in
      M' \otimes_{\max} A$. 
      %Let $\pi : M' \otimes_{\max} A \to B(H)$ be 
      Using $z\in M' \otimes 1$ (for which  $D(z)=0$) we find $T\in M''=M$ and using
      $z\in 1 \otimes A$ we find $\d(a)=[a,T]$, i.e. $\d=\d_T$.
      This proves $2\inf\{\|T\| \mid T\in M, \d=\d_T\} \le  \|D\|_{cb} \le  \|\d\|_{mb}$.
      Conversely, if $T\in M, \d=\d_T$ we clearly have $\|\d\|_{mb} \le 2 \|T\|$,
      and hence $\|\d\|_{mb} \le 2 \inf\{\|T\| \mid T\in M, \d=\d_T\}$.
     \end{proof}
     \def\t{\theta}
     \begin{cor} In the situation of the preceding theorem, 
     Let $\theta \in B(H)$ be such that
     $\d_\t (A) \subset M$ with $ \|\d_\t\|_{mb}  <\infty$. Then
     $\t\in M+A' $ and
     $$\|\d_\t\|_{mb} = 2\inf \{   \|T\| \mid  T\in M, \   \t-T \in A' \}.$$
     \end{cor}
     By \eqref{e13} the derivation $\d: C \to N=B(H) \bar\otimes M$ described in the first part satisfies
     $$   \sqrt n/2\le  \|\d\|_{mb}\le 2 (\sqrt n+1).$$
     \def\D{\Delta}
     
     \section{An alternate proof}
     
    We end this note with a quick direct ``conceptual" proof that $   (1+\sqrt n)/8\le  \|\d\|_{mb}$.\\
     Recall that $N=B(H)\bar\otimes M$.
     We already observed that
     $\|\d\|_{mb}$ is the same as the mb-norm
     of the derivation $\d_\D : \sigma(C)\otimes 1 \to N$.
     By the extension property  of mb-maps (see \cite[p. 266]{100} or  \cite{157})
    $\d_\D$ admits  an extension $d: B(H) \otimes 1 \to N$ with
    $\|d\|_{mb} =  \|\d\|_{mb}$. Now observe that
    $d_\D(x_j \otimes 1) = -P_\Omega \otimes x_j$, and the subspace $F=P_\Omega \otimes E$ (spanned
    by $\{ P_\Omega \otimes x_j\}$) admits a projection $Q: N \to F$
    with $\|Q\|_{cb} \le 2$. Namely
    we can take $Q= p \otimes P$ with $P$ as before
    and $p$ a completely contractive Hahn-Banach projection  onto $\CC P_\Omega $. Thus
    the map $Qd: B(H) \otimes 1 \to F$ has cb norm $\le 2$, and it takes
    $x_j \otimes 1$ to $-P_\Omega \otimes x_j$.
    Let $v_1: E \to B(H) \otimes 1$ 
    and $v_2: P_\Omega \otimes E \to E$ be the obvious (completely isometric) maps.
    Then the identity of $E$ can be rewritten as
    $  (v_2Qd )(v_1)$ which is a factorization through $B(H)$. But it is known (see \cite[Th. 4.1]{HP2}) that the cb-factorization constant
     of the identity of the space $E$
       through any $B(H)$ is $\ge (1+\sqrt n)/4$ (the extra factor 2 comes from the comparison of $E$ with the space $R_n \cap C_n$ from \cite[Th. 4.1]{HP2}). 
     Thus we recover
     $$(1+\sqrt n)/4 \le \|  v_2Qd \|_{cb} \|  v_1 \|_{cb} \le \| Q \|_{cb} \|  d  \|_{cb}\le  2\|  d  \|_{cb}  \le 2\|  \d  \|_{mb} $$ and hence $   (1+\sqrt n)/8\le  \|\d\|_{mb}$.
      \def\b{\beta}
      \def\a{\alpha}
     \def\M{{\cal M}}
          \begin{rem} 
          Let $(x_j)$ be as before. Let $E_n=span[x_j] \subset B(H)$.
          Let $\M\subset B(K)$ be a von Neumann algebra such that
          for each $n$ there is a cb-factorization of the identity of $E_n$          through $\M$,  
           with factorization constant $o(\sqrt n)$. 
           By this we mean that we can find maps $\a_n: E_n \to \M$
            and $\b_n : \M \to E_n$ such that
            $$\|\a_n  \|_{cb}   \| \b_n \|_{cb}   =o(\sqrt n).$$
            This holds in particular if $M$ embeds in $\M$. It probably holds for any non-injective $\M$.
            
        Let $A_n$ be either the $C^*$-algebra generated by $E_n$  or $A_n=C^*_\lambda(\F_n)$.
          Then we can find a derivation
          $\d_n : A_n \otimes 1 \to B(H) \bar\otimes \M $
          relative to the embedding
          $$  A_n \otimes 1 \to B(H) \bar\otimes B(K)$$
          such that
          $${ \|\d_n\|_{mb } / \|\d_n\|_{cb } } \to \infty.$$ 
          Indeed, it is easy to ``transplant" the   example from \S \ref{s1}.
          \end{rem}
       \n\textbf{Acknowledgement.} I am very grateful to Jean Roydor
    for stimulating and useful conversations.

     \vfill\eject
     %$$***$$

    \centerline{\bf Appendix}  
    
    \bigskip
    
    A close inspection  shows that the error  in \cite{[19]} lies in an incorrect identification.
    Using the notation of \cite{[19]}
    the identity
    \begin{equation}  \label{e5}
    TBT^{-1}/J= S(B/J)S^{-1} \end{equation}
    is incorrect.
    
    Indeed,  we have
    $$B=\{pb +(1-p)b \mid b\in B\} \subset B^{**} \subset B(H)$$
    with
    $$H=pH \oplus (1-p) H$$
    and 
    $$TBT^{-1} \subset B(H),$$
    which can be described as follows:
    $$TBT^{-1} =\{ S(pb)S^{-1} + (1-p) b \mid b\in B\}= \{ y_1+y_2 \mid S^{-1}(y_1)S +y_2 \in B 
    , \ y_1\in pB^{**}, y_2\in (1-p)B^{**} \}.$$
    Indeed, 
    assuming
    $y_1\in pB^{**}, y_2\in (1-p)B^{**} $,
    if $b= S^{-1}(y_1)S +y_2$ then $pb=S^{-1}(y_1)S $ and $ (1-p) b= y_2$ so that
  $y_1+y_2 =S(pb)S^{-1} + (1-p) b$. This proves ``$\supset$". The reverse is clear since
  $pB \subset pB^{**}$ and $(1-p) B \subset (1-p) B^{**}$.
  
  Note
  $$J=(1-p)J=\{ (1-p)j \mid j\in J\} .$$

 Since $B/J\simeq pB$
  the
identity \eqref{e5} must be interpreted as $ TBT^{-1}/J= S(pB)S^{-1}$. This
implies 
$$ TBT^{-1}\subset S(pB)S^{-1} + J,$$
and hence after the action of $  T^{-1} (.)T$ 
$$B\subset pB + J$$
which implies
$$(1-p) B \subset (1-p) J  =J.$$
However, in general we have $  (1-p) B \not= J$.
To verify that, assume $  (1-p) B  = J$. 
Let $f: B \to J $ be defined by $f(x)=   (1-p) x .$
Then $f$ is a $*$-homomorphism and a linear projection
from $B$ to $(1-p) B  = J$. In otherwords we have a splitting
of $B$ as $B/J\oplus J$, which of course fails in general.

  \vfill\eject
      
\end{document}